\documentclass[11pt]{amsart}
\usepackage{amssymb,amsmath,amsthm,amscd}
\usepackage{leftidx}
\usepackage{graphicx}
\usepackage{color}
\usepackage{microtype}
\usepackage{hyperref}
\usepackage[usenames,dvipsnames,svgnames,table]{xcolor}
\usepackage{tikz}
\usepackage{comment}
\usepackage{epstopdf} 
\usepackage{float}
\usepackage{tikz-cd}
\usepackage[all]{xy}
\newtheorem{lemma}{Lemma}[section]
\newtheorem{theorem}{Theorem}

\newtheorem{proposition}[lemma]{Proposition}

\theoremstyle{definition}

\theoremstyle{remark}
\newtheorem{remark}{Remark}
\newcommand{\C}{\mathbb{C}}
\newcommand{\D}{\mathbb{D}}
\newcommand{\N}{\mathbb{N}}

\newcommand{\R}{\mathbb{R}}
\newcommand{\Z}{\mathbb{Z}}

\newcommand{\cC}{\mathcal{C}}

\DeclareMathOperator{\re}{Re}
\DeclareMathOperator{\im}{Im}

\renewcommand{\epsilon}{\varepsilon}
\renewcommand{\phi}{\varphi}

\DeclareMathOperator{\resit}{r\acute esit}

\newcommand{\actloc}{\cC^{\mathrm{bif}}}

\date{\today}

\begin{document}

\title[On The Support of The Bifurcation Measure]{On The Support of The Bifurcation Measure Of Cubic Polynomials}

\author[H.~Inou]{Hiroyuki Inou}
\address{Department of Mathematics, Kyoto University, Kyoto 606-8502, Japan}
\email{inou@math.kyoto-u.ac.jp}

\author[S.~Mukherjee]{Sabyasachi Mukherjee}
\address{Institute for Mathematical Sciences, Stony Brook University, NY, 11794, USA}
\email{sabya@math.stonybrook.edu}

\thanks{The first author was supported by JSPS KAKENHI Grant Numbers 26400115 and 26287016.}

\maketitle

\begin{abstract}
We construct new examples of cubic polynomials with a parabolic fixed point that cannot be approximated by Misiurewicz polynomials. In particular, such parameters admit maximal bifurcations, but do not belong to the support of the bifurcation measure.  
\end{abstract}

\tableofcontents

\section{Introduction}\label{sec_intro}

The connectedness locus $\cC$ of cubic polynomials is the set of parameters such that the corresponding Julia sets are connected. It is a compact set in the parameter space $\C^2$ of all cubic polynomials \cite{BH1}. For suitable parametrizations, the two critical points of a cubic polynomial can be holomorphically followed throughout the parameter space (see \cite{BH1,DF,Du1} for various related parametrizations). 

In \cite{Du1}, cubic polynomials were parametrized as\footnote{The chosen parametrization plays no special role in the current paper as the results of this paper are mostly coordinate-free.} $$f_{c,v}(z)=z^3- 3c^2z+ 2c^3 +v,$$ where $c,v\in\C$. The two critical points of $f_{c,v}$ are $\pm c$. The critical point $\pm c$ is said to be \emph{passive} near the parameter $(c_0,v_0)$ if the sequence of functions $(c,v)\mapsto f_{c,v}^{\circ n}(\pm c)$ forms a normal family in a neighborhood of $(c_0,v_0)$ in $\C^2$. Otherwise, $\pm c$ is said to be \emph{active} near $(c_0,v_0)$. According to \cite{Du1}, the critical point $\pm c$ is active precisely on the set $\partial\cC^{\pm}$, where $\cC^{\pm}$ is the set of parameters for which $\pm c$ has bounded orbit. Note that $\cC=\cC^+\cap\cC^-$. 

The \emph{bifurcation locus} $\actloc$ of cubic polynomials is defined as the complement of the set of all $J$-stable parameters; i.e. $\actloc$ consists of parameters for which at least one critical point is active (see \cite[Theorem 4.2]{Mcm1} for several equivalent conditions for $J$-stability). Clearly, we have that $\actloc=\partial\cC^+\cup\partial\cC^-\supset\partial\cC$. 

We denote by $\cC^*$ the intersection of the activity loci of the two critical points; i.e. $\cC^*:=\partial\cC^+\cap\partial\cC^-\subset\partial\cC$. Since $\cC^*$ is the set of parameters for which both critical points are active, it is called the \emph{bi-activity locus}.

DeMarco introduced a natural $(1,1)$-current supported exactly on the bifurcation locus \cite{DeM01, DeM03}. In \cite{BB}, Bassanelli and Berteloot constructed a natural probability measure supported on the boundary of the connectedness locus $\partial\cC$ (which is contained in the bifurcation locus), by taking a power of the bifurcation current. This measure is called the \emph{bifurcation measure}, and is denoted by $\mu_{\mathrm{bif}}$. It has several interesting dynamical properties, and can be thought of as the correct generalization of the harmonic measure of the Mandelbrot set. 
Dujardin and Favre \cite{DF} showed that the support of the bifurcation measure is equal to the closure of Misiurewicz parameters (in fact, Misiurewicz parameters are equidistributed by $\mu_{\mathrm{bif}}$), which is a subset of $\cC^*$. 

The bifurcation measure is designed to capture maximal bifurcations in the family. In this vein, one may ask if $\mathrm{Supp}(\mu_{\mathrm{bif}})$ is equal to $\cC^*$. However, it was pointed out by Douady that there are parabolic parameters in $\cC^*$ that cannot be approximated by Misiurewicz parameters \cite[Example 6.13]{DF}. These maps have a parabolic fixed point which attracts the forward orbits of both critical points. Consequently, they have a complex one-dimensional (quasi-conformal) deformation space. Moreover, any small perturbation of such a map is either parabolic (and quasi-conformally conjugate to the original map) or has at least one attracting fixed point. In other words, these parabolic parameters are \emph{parabolic-attracting} in the language of \cite{Adam1} (or \emph{virtually attracting} in the language of \cite{Buff1}).

This naturally leads to a study of the set $\cC^*\setminus\mathrm{Supp}(\mu_{\mathrm{bif}})$. The principal goal of this paper is to construct examples of parabolic-repelling parameters in $\cC^*\setminus\mathrm{Supp}(\mu_{\mathrm{bif}})$ (at a first glance, it is much less obvious that such parameters may lie outside the support of the bifurcation measure), which shows that the gap between the bi-activity locus and the support of the bifurcation measure is bigger than what was known previously.

\begin{theorem}[Parabolic-Repelling Parameters outside $\mathrm{Supp}(\mu_{\mathrm{bif}})$]\label{main_thm}
There exists an interval $I$ of parabolic-repelling parameters such that $\displaystyle I~\subset~\cC^*~\setminus~\mathrm{Supp}(\mu_{\mathrm{bif}})$. More precisely, if $a\in I$, then any sufficiently small perturbation of $f_a$ (in the cubic parameter space) is either in the escape locus, or has an attracting/parabolic fixed point.
\end{theorem}

The interval $I$ satisfying the statement of Theorem~\ref{main_thm} consists of parameters having a parabolic fixed point of multiplier $1$. In other words, $I$ is contained in the slice $\mathrm{Per}_1(1)$ (see the definition below). 

Note that every parabolic-attracting parameter is contained in $\cC^*~\setminus~\mathrm{Supp}(\mu_{\mathrm{bif}})$. On the other hand, by employing parabolic perturbation techniques, we show that a (suitably chosen) parabolic-repelling parameter can be approximated by Misiurewicz parameters only if two dynamically defined conformal invariants associated with the map satisfy a certain condition. The interval $I$ in Theorem~\ref{main_thm} is concocted so that the two conformal invariants of the corresponding maps violate this condition.

Every parameter in $\partial\cC$ near $I$ has an attracting or parabolic cycle. Moreover, each of these nearby parameters admits a disk of quasi-conformal deformations. On the other hand, parameters outside $\cC$ close to $I$ have at least one escaping critical point. Hence, by the usual wringing deformation (see \cite{BH1}), such parameters also admit quasi-conformal deformations. It follows that the interval $I$ does not intersect the closure of quasi-conformally rigid parameters.

According to \cite[Proposition 6.14]{DF}, the support $\mathrm{Supp}(\mu_{\mathrm{bif}})$ of the bifurcation measure is the Shilov boundary of $\cC$. Heuristically speaking, the interval $I$ that we construct in Theorem~\ref{main_thm} does not lie in the Shilov boundary of $\cC$ since $\partial\cC$ is foliated by holomorphic disks (coming from quasi-conformal deformations described above) locally near $I$.

Let us now describe the organization of the paper and the key ideas of the proof of the main theorem. 

In Section~\ref{slice}, we recall some basic properties of cubic polynomials with a parabolic fixed point of multiplier $1$. These maps form a complex one-dimensional slice (of the parameter space $\C^2$ of all cubic polynomials) which is denoted by $\mathrm{Per}_1(1)$. We focus on `real' maps in $\mathrm{Per}_1(1)$ for which the two critical points are complex conjugates of each other and both critical orbits converge to the unique parabolic fixed point. Both critical points of the maps under consideration are active; i.e. these maps belong to the bi-activity locus $\cC^*$. We associate a `global' conformal conjugacy invariant (called the \emph{critical Ecalle height}) with these maps, that can be used as a local parameter for the quasi-conformal deformation space of these maps. We also recall the notion of the residue fixed point index of a parabolic map, which is a `local' conformal conjugacy invariant of parabolic dynamics. These two invariants play a crucial role in the proof of the main theorem.

Section~\ref{para_perturb} contains a brief discussion of perturbation of parabolic points. We go over some basic properties of persistent Fatou coordinates, horn maps, and lifted phase for perturbations of parabolic maps.

The final Section~\ref{proof_main_thm} is devoted to the proof of Theorem~\ref{main_thm}. In Lemma~\ref{height_para_rep}, we show that if the critical Ecalle height of a real map in $\mathrm{Per}_1(1)$ (introduced in Section~\ref{slice}) is not too large, then the map is parabolic-repelling. Finally, a careful analysis of the lifted phase of the perturbed maps shows that if the critical Ecalle height of a parabolic-repelling parameter is not too small (i.e. bounded below by a function of the residue fixed point index), every perturbation of such a map either has an attracting fixed point or has an escaping critical point. This yields an interval of real parabolic-repelling parameters in $\mathrm{Per}_1(1)$ that cannot be approximated by Misiurewicz parameters, and completes the proof of the theorem.

\section{The Slice $\mathrm{Per}_1(1)$}\label{slice}

The family of cubic polynomials with a parabolic fixed point of multiplier $1$ is denoted by $\mathrm{Per}_1(1)$, and can be parametrized as $$\mathrm{Per}_1(1):=\{f_a(z)=z+az^2+z^3:a\in\C\}.$$ This family has been studied in \cite{Pas}. We only recall some basic facts about $\mathrm{Per}_1(1)$ that we will need in this paper.
 
If $a\in\R$, then $f_a$ commutes with the complex conjugation map. It is easy to see that for $a\in(-\sqrt{3},\sqrt{3})$, the two critical points of $f_a$ are complex conjugate. We denote the critical point in the lower (respectively, upper) half plane by $c_{-}(a)$ (respectively, $c_{+}(a)$). 
 
For $a=0$, the corresponding map has a double parabolic fixed point at the origin (i.e. it has two attracting petals), while for $a\neq0$, the corresponding map has a simple parabolic fixed point at the origin. The two immediate basins of $f_0$ are real-symmetric, and each basin contains a critical point.

It follows by real symmetry that for $a\in(-\sqrt{3},\sqrt{3})\setminus\{0\}$, both critical points of $f_a$ lie in the unique immediate basin of the parabolic fixed point. In particular, the parabolic basin is connected, and the Julia set is a Jordan curve.

For $a=\pm\sqrt{3}$, the two critical points of $f_a$ coalesce to form a double critical point. Finally, if $a\in(-\infty,-\sqrt{3})\cup(\sqrt{3},+\infty)$, then the two critical points of $f_a$ are real.

Let $\mathcal{I}:=(0,\sqrt{3})$. For all $a\in\mathcal{I}$, we will normalize the attracting (respectively, repelling) Fatou coordinate $\psi^{\mathrm{att}}_a$ (respectively, $\psi^{\mathrm{rep}}_a$) of $f_a$ at the parabolic fixed point $0$ such that $\psi^{\mathrm{att/rep}}_a$ commutes with complex conjugation. Since Fatou coordinates are unique up to addition of a complex constant, the above normalization implies that $\psi^{\mathrm{att/rep}}_a$ is unique up to horizontal translations. Therefore, the imaginary part of $\psi^{\mathrm{att/rep}}_a$ is well-defined. We will refer to $\im(\psi^{\mathrm{att/rep}}_a)$ as the attracting/repelling Ecalle height.

In particular, we have that $$\re(\psi^{\mathrm{att}}_a(c_{+}(a)))=\re(\psi^{\mathrm{att}}_a(c_{-}(a))),\ \mathrm{and}\ \im(\psi^{\mathrm{att}}_a(c_{\pm}(a)))=\pm h_a/2,$$ for some $h_a>0$. Moreover, $$h_a=\psi^{\mathrm{att}}_a(c_{+}(a))-\psi^{\mathrm{att}}_a(c_{-}(a))$$ is a conformal conjugacy invariant of $f_a$. We call $h_a$ the \emph{critical Ecalle height} of $f_a$.

By changing $h_a$ using a quasi-conformal deformation argument, we will show that all maps on $\mathcal{I}$ are quasi-conformally conjugate. 

\begin{proposition}\label{para_arc}
All cubic polynomials $f_a$, where $a\in\mathcal{I}$, are quasi-conformally conjugate. Moreover, $\mathcal{I}$ admits a real-analytic parametrization $a:(0,+\infty)\to~\mathcal{I}$ such that $h_{a(t)}=t$.
\end{proposition}
\begin{proof}
The proof is similar to \cite[Theorem 3.2]{MNS}.

Pick $a_0\in\mathcal{I}$ such that $h_{a_0}=t_0$. We can choose the attracting Fatou coordinate $\psi_{a_0}^{\mathrm{att}}$ so that both the critical points $c_{\pm}(a_0)$ have real part $1/2$ within the Ecalle cylinder.

Let $\zeta=x+iy$. Now, for every $t\in(0,+\infty)$, the map 
\[
 \ell_t : (x,y)\longmapsto \left(x,\frac{t}{t_0}y\right)
\]

is a quasi-conformal homeomorphism of $\C/\Z$ commuting with complex conjugation. Note that $\ell_t(1/2,\pm t_0/2) = (1/2,\pm t/2)$. Translating the map $\ell_t$ by positive integers, we obtain a quasi-conformal map $\ell_t$ commuting with $\zeta\mapsto\zeta+1$ on a right half plane. 

By the coordinate change $\psi_{a_0}^{\mathrm{att}}:z\mapsto \zeta$, we can transport this Beltrami form (defined by the quasi-conformal homeomorphism $\ell_t$) into the attracting petal at $0$ such that it is forward invariant under $f_{a_0}$. Pulling it back by the dynamics, we can spread the Beltrami form to the entire parabolic basin. Extending it by the zero Beltrami form outside of the parabolic basin, we obtain an $f_{a_0}$-invariant Beltrami form. Moreover, this Beltrami form respects the complex conjugation map.

The Measurable Riemann Mapping Theorem (with parameters) now supplies a quasi-conformal map $\phi_t$ integrating this Beltrami form such that $\phi_t$ commutes with complex conjugation. We can normalize $\phi_t$ such that it fixes $0$ and $\infty$. Then, $\phi_t$ conjugates $f_{a_0}$ to a cubic polynomial fixing the origin. We can further require that the conjugated polynomial $\phi_t\circ f_{a_0}\circ\phi_t^{-1}$ is monic. By \cite{Na}, $\phi_t\circ f_{a_0}\circ\phi_t^{-1}$ has a simple parabolic fixed point of multiplier $1$ at the origin. Hence, $\phi_t\circ f_{a_0}\circ\phi_t^{-1}\in\mathrm{Per}_1(1)$.

Since $\phi_t$ commutes with complex conjugation, it follows that $\phi_t\circ f_{a_0}\circ\phi_t^{-1}$ is a real cubic polynomial. Furthermore, since the complex conjugation map exchanges the two distinct critical points of $f_{a_0}$, the same must be true for $\phi_t\circ f_{a_0}\circ\phi_t^{-1}$ as well. It follows that $\phi_t\circ f_{a_0}\circ\phi_t^{-1}=f_{a(t)}$, for some $a(t)\in\mathcal{I}$.

The attracting Fatou coordinate of $f_{a(t)}$ is given by $\psi_{a(t)}^{\mathrm{att}}= \ell_t\circ\psi_{a_0}^{\mathrm{att}}\circ\phi_t^{-1}$. Thus, $\im\psi^{\mathrm{att}}_{a(t)}(c_{\pm}(a(t)))=\pm t/2$, and hence $h_{a(t)}=t$.

Note that the Beltrami form constructed above depends real-analytically on $t$, so the parameter $a(t)$ depends real-analytically on $t$ as well. Therefore, we obtain a real-analytic map $a:(0,+\infty)\to\mathcal{I}$. Since the critical
points of all $f_{a(t)}$ have different Ecalle heights, which is a conformal invariant, this map is injective. 

It remains to show that $a((0,+\infty))=\mathcal{I}$. As $t\to0$, the two critical points of $f_{a(t)}$ tend to merge together. It follows that $\displaystyle\lim_{t\to0^+}a(t)=\sqrt{3}$. On the other hand, any accumulation point of $a(t)$ as $t\to+\infty$ must be a double parabolic parameter. Hence, $\displaystyle\lim_{t\to+\infty}a(t)=0$. Since $a((0,+\infty))$ is connected, it follows that $a((0,+\infty))=(0,\sqrt{3})=\mathcal{I}$.
\end{proof}

The next lemma shows that the arc $\mathcal{I}$ is contained in the bi-activity locus $\cC^\ast$.

\begin{lemma}\label{para_in_bif_locus}
$\mathcal{I}\subset\cC^\ast$.
\end{lemma}
\begin{proof}
Let $\tilde{a}\in\mathcal{I}$. For cubic polynomials $g$ close to $f_{\tilde{a}}$, we mark the two critical points of $g$ by $c_\pm(g)$. 

We will prove the lemma by contradiction. To this end, let us assume that there is an open neighborhood $U$ of $\tilde{a}$ (in the full cubic parameters space $\C^2$) such that the sequence of holomorphic functions $\{U\ni g\mapsto g^{\circ n}(c_+(g))\}_{n\in\N}$ forms a normal family. Note that $U$ intersects a hyperbolic component of period one non-trivially such that for these hyperbolic polynomials, both critical orbits converge to a common attracting fixed point. By normality of the above family of functions, it follows that the forward orbit of $c_+(g)$ converges to a fixed point $w(g)$ (of $g$) for every map $g$ in $U$. Moreover, $w(g)$ is a holomorphic function of $g$ on $U$ (as it is a limit of holomorphic functions). Therefore, the multiplier $g'(w(g))$ of the fixed point is also a holomorphic function of $g$ in $U$. Since the multiplier of the fixed point of $f_{\tilde{a}}$ is $1$, it follows by the maximum modulus principle that $w(g)$ must be a repelling fixed point for an open set of maps in $U$. However, this is impossible as an orbit cannot non-trivially converge to a repelling fixed point. 

Since $\tilde{a}$ is real and $c_\pm(f_{\tilde{a}})$ are complex conjugate,
$c_+(f_{\tilde{a}})$ is active if and only if $c_-(f_{\tilde{a}})$ is.

Hence, both critical points of $f_{\tilde{a}}$ are active. Thus, $\mathcal{I}\subset\cC^\ast$.

\end{proof}

The \emph{residue fixed point index} of $f_a$ at the parabolic fixed point $0$ is defined to be the complex number
\begin{align*}
\displaystyle \iota(f_a,0) &= \frac{1}{2\pi i} \oint \frac{dz}{z-f_a(z)},
\end{align*}
where we integrate in a small loop in the positive direction around $0$. A simple computation shows that $\iota(f_a,0) =1/a^2$ (when $a\neq0$). The fixed point index is a conformal conjugacy invariant (see \cite[\S 12]{M1new} for a general discussion on the concept of residue fixed point index). 

The \emph{r{\'e}sidu it{\'e}ratif} of the parabolic fixed point $0$ of $f_a$ is defined as $1-\iota(f_a,0)=1-1/a^2$. It is denoted by $\resit(f_a)$. This quantity plays an important role in the study of perturbation of parabolic germs.

The origin is called a parabolic-attracting (respectively, parabolic-repelling) fixed point of $f_a$ if $\re(\resit(f_a))<0$ (respectively, if $\re(\resit(f_a))>0$). Clearly, for $a\in(0,1)$, the origin is a parabolic-attracting fixed point of $f_a$. On the other hand, for $a\in(1,\sqrt{3})$, the origin is a parabolic-repelling fixed point of $f_a$.

\begin{lemma}[Upper Bound on $\resit(f_a)$]\label{upper_bound}
For all $a\in\mathcal{I}$, we have that $\resit(f_a)<2/3$.
\end{lemma}
\begin{proof}
If $a\in\mathcal{I}=(0,\sqrt{3})$, then $1/a^2>1/3$. Hence, $\resit(f_a)=(1-1/a^2)<2/3$.
\end{proof}

\section{Perturbation of Parabolic Points}\label{para_perturb}

In this section, we will recall some basic facts on perturbation of parabolic points, and fix the terminologies for the rest of the paper.

If a map $f_a$ (with $a\in\mathcal{I}$) is perturbed outside of $\mathrm{Per}_1(1)$, the simple parabolic fixed point $0$ bifurcates into two simple fixed points each of which has multiplier close to $1$. We will only be concerned with perturbations for which $\arg(\lambda-1)\in\left[\frac{\pi}{4},\frac{3\pi}{4}\right]\cup\left[\frac{5\pi}{4},\frac{7\pi}{4}\right]$, where $\lambda$ is the multiplier of one of the two bifurcating fixed points (other perturbations always create an attracting fixed point, and are uninteresting for our purpose, see \cite[\S 3.1]{Shi} for further details).

Let us now briefly describe the dynamics near the bifurcated fixed points for the class of perturbed maps under consideration. In the dynamical plane of such a perturbed map, there is a curve joining these two fixed points, which is called the \emph{gate}. Moreover, there exist an attracting domain $V^\mathrm{att}$, and a repelling domain $V^\mathrm{rep}$ having the two simple fixed points on their boundaries. The points in the attracting domain eventually transit through the gate, and escape to the repelling domain. Moreover, there are injective holomorphic maps $\psi^\mathrm{att/rep}$ defined on $V^\mathrm{att/rep}$ conjugating the dynamics to translation by $+1$ (as long as the orbit stays in the domain of definition of the maps). The maps are referred to as persistent Fatou coordinates. The quotient of $V^\mathrm{att/rep}$ by the dynamics is a bi-infinite cylinder, which is denoted by $C^\mathrm{att/rep}$.

The class of perturbed maps considered above are said to exhibit \emph{eggbeater dynamics} (see \cite[Figures~4,5]{Shi}).

There exists an open set $U$ in the cubic parameter space with $\mathcal{I}\subset\partial U$ such that every map in $U$ exhibits eggbeater dynamics and admits persistent Fatou coordinates as above (compare \cite[Proposition 3.2.2, Proposition 3.2.3]{Shi}). We will only consider perturbations of $f_a$ in $U$. It makes sense to label the critical points of the perturbed maps as $c_\pm$ so that $\im(c_+)>0$, and $\im(c_-)<0$.
 
The \emph{lifted horn maps} of the parabolic fixed point of $f_a$ are defined as $H_a^\pm=\psi_a^{\mathrm{att}}\circ\left(\psi_a^{\mathrm{rep}}\right)^{-1}$ on regions with sufficiently large imaginary part in the repelling Fatou coordinates. More precisely, $H_a^+$ (respectively, $H_a^-$) is defined on $\lbrace\im(\zeta)>M\rbrace$ (respectively, on $\lbrace\im(\zeta)<-M\rbrace$) for some sufficiently large positive $M$. 

By our normalization of Fatou coordinates, we have that $$\psi_{a}^{\mathrm{att}}(z)-\psi_{a}^{\mathrm{rep}}(z)\approx\mp i\pi \resit(f_a),$$ as $z$ tends to the upper/lower end of the cylinders. It follows that $$H_a^\pm(\zeta)\approx \zeta\mp  i\pi \resit(f_a)$$ as $\im(\zeta)\to\pm\infty$. The map $\mathrm{exp}:\zeta\mapsto e^{2\pi i\zeta}$ conjugates the lifted horn maps $H_a^\pm$ to a pair of germs $\mathfrak{h}_a^\pm$ fixing $0$ and $\infty$ respectively. These maps are called the \emph{horn maps} of the parabolic fixed point of $f_a$. They satisfy $$\left(\mathfrak{h}_a^+\right)'(0)=e^{2\pi^2\resit(f_a)}= \left(\mathfrak{h}_a^-\right)'(\infty).$$

For perturbed maps, one can still define horn maps. Points in $V^\mathrm{rep}$ with large imaginary part in the repelling Fatou coordinate eventually land in $V^\mathrm{att}$. This defines a map from the ends of $\psi^\mathrm{rep}(V^\mathrm{rep})$ to $\psi^\mathrm{att}(V^\mathrm{att})$, which is called the lifted horn map of the perturbed map. The persistent Fatou coordinates can be normalized so that they depend continuously on the parameters. With such normalizations, the horn maps depend continuously on the parameters.

In the perturbed situation, points in $V^\mathrm{att}$ are mapped to $V^\mathrm{rep}$ by some large iterate of the dynamics (where the required number of iterations tends to $+\infty$ as the perturbation goes to zero). This \emph{transit map} induces an isomorphism of the cylinder $\C/\Z$ (via the Fatou coordinates). Hence, the transit map can be written as $\left(\psi^\mathrm{rep}\right)^{-1}\circ T_\sigma\circ\psi^{\mathrm{att}}$, where $T_\sigma$ is translation by some complex number $\sigma$. The complex number $\sigma$ is called the \emph{lifted phase} of the perturbed map.

Finally, one can define a \emph{return map} from the top and bottom ends of $V^\mathrm{rep}$ to $V^\mathrm{rep}$ for the perturbed maps. In the Fatou coordinates, this map can be expressed as the composition of the lifted horn map and the translation $T_\sigma$ (for some $\sigma\in\C$). For sufficiently small perturbations, $\mathrm{exp}:\zeta\mapsto e^{2\pi i\zeta}$ conjugates these return maps to germs $\mathcal{R}^\pm$ that are close to $e^{2\pi i\sigma}\mathfrak{h}_a^\pm$.

\section{Proof of Theorem~\ref{main_thm}}\label{proof_main_thm}

Let $a\in(0,1)$. Then $0$ is a parabolic-attracting fixed point of $f_a$. By \cite[Theorem 1]{Buff1} (also compare \cite[Theorem 12.10]{M1new}), every cubic polynomial sufficiently close to $f_a$ has at least one non-repelling fixed point. Thus, $f_a$ cannot be approximated by Misiurewicz maps. Following \cite[Example 6.13]{DF}, one can conclude that the parabolic-attracting map $f_a$ lies in $\cC^*\setminus\mathrm{Supp}(\mu_{\mathrm{bif}})$.

The goal of this section is to prove Theorem~\ref{main_thm}, which asserts that $\cC^*\setminus\mathrm{Supp}(\mu_{\mathrm{bif}})$ does not consist only of parabolic-attracting maps, it contains parabolic-repelling maps as well (which is perhaps more surprising). To this end, we need to study the geometry of the dynamical plane of maps in $\mathcal{I}$. 

\begin{lemma}\label{annulus_modulus}
Let $a\in\mathcal{I}$. In the dynamical plane of $f_a$, the projection of the basin of infinity into the repelling Ecalle cylinder is an annulus of modulus $\frac{\pi}{\ln 3}$.
\end{lemma}
\begin{proof}
The proof essentially follows the arguments of \cite[Proposition~7]{HSS}, we adapt the proof for our setting. 

Since the Julia set of $f_a$ is connected, the projection of its basin of infinity into the repelling Ecalle cylinder is an annulus. Under the (correctly normalized) B{\"o}ttcher coordinate of $f_a$, which is a conformal isomorphism from the basin of infinity onto $\widehat{\C}\setminus\overline{\D}$, the access of the basin of infinity to the parabolic fixed point $0$ corresponds to the fixed point $1$ of $\zeta^3$. Moreover, under the B{\"o}ttcher map, the aforementioned annulus corresponds to the annulus obtained by gluing the two boundary curves in $\widehat{\C}\setminus\overline{\D}$ of the shaded region shown in Figure~\ref{modulus_pic}(Left) under the dynamics $\zeta\mapsto\zeta^3$. 

\begin{figure}[ht!]
\begin{center}
\includegraphics[scale=0.42]{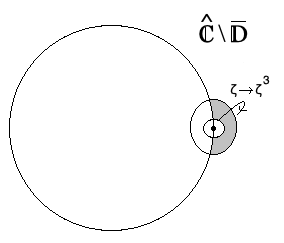}\ \includegraphics[scale=0.36]{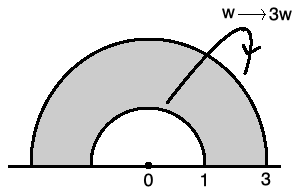}\ \includegraphics[scale=0.42]{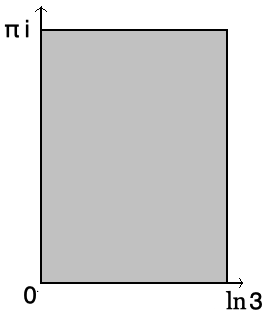}
\caption{Left: Under the B{\"o}ttcher coordinate, the basin of infinity in the repelling cylinder of $f_a$ maps to the shaded annulus where the indicated boundary curves are identified by $\zeta\mapsto\zeta^3$. Middle: Under the Koenigs linearizing coordinate, the annulus in the left figure maps to the shaded annulus where the indicated boundary curves are identified by $w\mapsto 3w$. Right: Under a suitable branch of logarithm, the annulus in the middle figure maps to the shaded rectangle where the vertical boundary curves are identified by translation by $\ln 3$.}
\label{modulus_pic}
\end{center}
\end{figure}

Under a Koenigs linearizing coordinate for the repelling fixed point $1$ of $\zeta^3$ (which has multiplier $3$), the annulus depicted in Figure~\ref{modulus_pic}(Left) corresponds to the annulus obtained by gluing the two boundary curves in the upper half-plane of the shaded region shown in Figure~\ref{modulus_pic}(Middle) under the dynamics $w\mapsto3w$. 

Finally under a suitable branch of logarithm, the annulus depicted in Figure~\ref{modulus_pic}(Middle) corresponds to the annulus obtained by gluing the two vertical boundary curves of the shaded rectangle shown in Figure~\ref{modulus_pic}(Right) under translation by $\ln3$. The modulus of the resulting annulus is clearly $\frac{\pi}{\ln 3}$. This completes the proof.

\end{proof}

Since $\frac{\pi}{\ln 3}>\frac{1}{2}$, it follows by \cite[Theorem I]{BDH} that the conformal annulus considered in Lemma~\ref{annulus_modulus} contains a round annulus of modulus at least $m=\frac{\pi}{\ln 3}-\frac{1}{2}\approx 2.3596$ (centered at the origin). Hence, due to real symmetry, there is an interval $(-m/2,m/2)$ of repelling Ecalle heights such that in the repelling Ecalle cylinder, the round cylinder $\R/\Z\times(-m/2,m/2)$ is contained in the projection of the basin of infinity (see Figure~\ref{cylinders_pic}). 

The proof of the main theorem makes essential use of the above geometric property of the basin of infinity (of the maps in $\mathcal{I}$) and bi-criticality of cubic polynomials. The rough idea of the proof is as follows. We construct a suitable sub-interval $I\subset\mathcal{I}$ (consisting of parabolic-repelling parameters) such that if the lifted phase of a nearby map has small imaginary part, then at least one critical point escapes to infinity (which is a consequence of the fact that the basin of infinity occupies a definite annulus). On the other hand, if the imaginary part of the lifted phase is large enough to prohibit a critical point from escaping, then there is an attracting fixed point. Thus, no map close to $I$ can be Misiurewicz.

\begin{figure}[ht!]
\begin{center}
\includegraphics[scale=0.32]{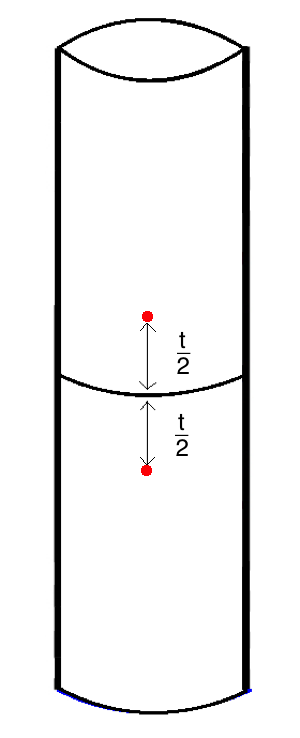}\hspace{2cm} \includegraphics[scale=0.325]{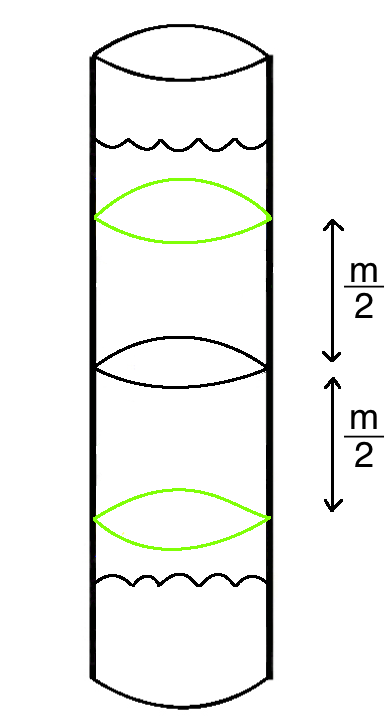}
\caption{Left: The attracting cylinder of $f_{a(t)}$ with the critical points marked. Right: The projection of the basin of infinity of $f_{a(t)}$ into repelling Ecalle cylinder contains the round cylinder $\R/\Z\times(-m/2,m/2)$.}
\label{cylinders_pic}
\end{center}
\end{figure}

For the rest of the paper, we fix an $\epsilon$ sufficiently small; for instance with $0<\epsilon<0.01$. Recall that the \emph{shift locus} is the set of maps with both critical points escaping.

\begin{lemma}\label{height_para_rep}
Let $0<t<m-2\epsilon$. Then $f_{a(t)}$ lies on the boundary of the shift locus.
In particular, $f_{a(t)}$ is parabolic-repelling; i.e. $a(t)\in(1,\sqrt{3})$. 
\end{lemma}

\begin{proof}
Let $0<t<m-2\epsilon$. We consider a small open set $U$ with $a(t)\in\partial U$ such that perturbing $a(t)$ in $U$ creates eggbeater dynamics. 

Note that the critical points and the Fatou coordinates depend continuously on the parameter throughout $U$. Hence, by shrinking $U$, we can assume that the imaginary part of $\psi^{\mathrm{att}}(c_+)$ is less than $(t/2+\epsilon/3)<(m/2-2\epsilon/3)$, and the imaginary part of $\psi^{\mathrm{att}}(c_-)$ is greater than $(-t/2-\epsilon/3)> (-m/2+2\epsilon/3)$. 

Moreover, the basin of infinity can not get too small when $a(t)$ is slightly perturbed (compare \cite[Theorem 5.1(a)]{D2}). Hence, by shrinking $U$ further, we can also assume that for every parameter in $U$, the round cylinder $\R/\Z\times\left(-m/2+\epsilon/3, m/2-\epsilon/3 \right)$ is contained in the projection of the basin of infinity into the repelling cylinder (note that in the repelling Ecalle cylinder of $a(t)$, the round cylinder $\R/\Z\times\left(-m/2, m/2 \right)$ is contained in the projection of the basin of infinity). 

Now consider the following perturbation of $f_{a(t)}$ in $U$:
\[
 g_\delta(z)= f_{a(t)}(z)+\delta \quad (\delta>0).
\]
Since $g_\delta$ is real, the transit map is a horizontal translation. In the dynamical plane of such a perturbed map, the critical orbits ``transit'' from the attracting Ecalle cylinder to the repelling cylinder and the imaginary parts of the Fatou coordinates are preserved in the process. Hence, by our choice of $U$, the critical points of $g_\delta$ have repelling Ecalle heights in $\left(-m/2+2\epsilon/3, m/2-2\epsilon/3 \right)$. Since $g_\delta$ is in $U$, any point in its repelling cylinder with repelling Ecalle height in $\left(-m/2+\epsilon/3, m/2-\epsilon/3 \right)$ is contained in the projection of the basin of infinity. Therefore, for such perturbations, both critical points lie in the basin of infinity. Hence, the perturbed maps $g_\delta$ lie in the shift locus, and both of the bifurcated fixed points of the perturbed maps are repelling. 

Therefore, $f_{a(t)}$ lies on the boundary of the shift locus, and is parabolic-repelling; i.e. $a(t)\in(1,\sqrt{3})$. 
\end{proof}

Suppose that $a(t)\in\mathcal{I}$ with $0<t<m-2\epsilon$. Consider a small (eggbeater-type) perturbation of $f_{a(t)}$ with associated transit map $T_\sigma$. It follows from the proof of Lemma~\ref{height_para_rep} that if the lifted phase $\sigma$ of the perturbed map is real, then both of its critical points escape to infinity. In the next lemma, we look at the other side of the story. More precisely, we study perturbed maps whose lifted phase $\sigma$ has a large imaginary part.

\begin{lemma}\label{phase_restriction}
If $\im(\sigma)>\pi\resit(f_{a(t)})$, then the perturbed map has an attracting fixed point. 
\end{lemma}
\begin{proof}
For sufficiently small perturbations, the absolute value of the multiplier of the `return map' $\mathcal{R}^+$ at the origin is close to 
\[
 \vert e^{2\pi i\sigma}\cdot\left(\mathfrak{h}_{a(t)}^+\right)'(0)\vert=\vert e^{2\pi i\sigma}\cdot e^{2\pi^2\resit(f_{a(t)})}\vert=e^{-2\pi(\im(\sigma)-\pi\resit(f_{a(t)}))}.
\]
Therefore, if $\im(\sigma)>\pi\resit(f_{a(t)})$, then $0$ is an attracting fixed point of $\mathcal{R}^+$. It follows that one of the simple fixed points of the perturbed map is attracting.
\end{proof}

\begin{proof}[Proof of Theorem~\ref{main_thm}]
Choose $a(t)\in\mathcal{I}$ with $$4\pi/3-m+2\epsilon<t<m-2\epsilon.$$ This is possible because $4\pi/3-m+2\epsilon\approx1.8292+2\epsilon<1.85$, and $m-2\epsilon\approx2.3596-2\epsilon>2.33$. By Lemma~\ref{height_para_rep}, $f_{a(t)}$ is parabolic-repelling.

Consider a perturbation of $f_{a(t)}$ in the connectedness locus.
If it is not eggbeater-type, then we have an attracting or parabolic fixed point.
So we may assume this is an eggbeater-type perturbation.
Let the transit map be $T_\sigma$, for some $\sigma\in\C$. We can assume that $\im(\sigma)>0$ (the other case is symmetric).

Note that $\psi^{\mathrm{att}}_{a(t)}(c_{-}(a(t)))=- t/2$. Under small perturbation, the imaginary part of the attracting Fatou coordinate of $\psi^{\mathrm{att}}(c_{-})$ lies in $(-t/2-\epsilon/2,-t/2+\epsilon/2)$. Since the perturbed map is in the connectedness locus, the critical point must not land in the basin of infinity after exiting through the gate. But for a small perturbation, the basin of infinity occupies at least the cylinder $\R/\Z\times(-m/2+\epsilon/2,m/2-\epsilon/2)$ in the repelling cylinder. Note that $(-t/2-\epsilon/2)>(-m/2+\epsilon/2)$. Hence, the imaginary part of the lifted phase must be large enough to push the `lower' critical point sufficiently up so that it avoids the basin of infinity; i.e. $$- t/2+\epsilon/2+\im(\sigma)\geq m/2-\epsilon/2;$$ $$\mathrm{or},\quad \im(\sigma)\geq m/2+t/2-\epsilon.$$

By our choice of $t$ and Lemma~\ref{upper_bound}, we have that $$m/2+t/2-\epsilon>2\pi/3>\pi\resit(f_{a(t)}).$$ This means that the imaginary part of $\sigma$ is larger than $\pi\resit(f_{a(t)})$; i.e. $\im(\sigma)>\pi\resit(f_{a(t)})$. But then Lemma~\ref{phase_restriction} forces the perturbed map to have an attracting fixed point.

Hence $I:=a((4\pi/3-m+2\epsilon,m-2\epsilon))\subset\mathcal{I}$ consists of parabolic-repelling parameters that are not contained in $\mathrm{Supp}(\mu_\mathrm{bif})$. By Lemma~\ref{para_in_bif_locus}, we conclude that $I\subset\cC^*\setminus\mathrm{Supp}(\mu_\mathrm{bif})$.

\end{proof}

\begin{remark}
\begin{enumerate}
\item In the last sentence of the proof of Theorem~\ref{main_thm}, one can avoid appealing to Lemma~\ref{para_in_bif_locus}. Indeed, by Lemma~\ref{height_para_rep}, the interval $I$ is on the boundary of the shift locus and this directly implies that $I$ lies in the bi-activity locus $\mathcal{C}^\ast$.

\item Similar techniques can be used to prove the existence of parabolic-repelling biquadratic polynomials (lying on the parabolic arcs of period one of the tricorn) outside the support of the bifurcation measure, compare \cite[Theorem 1.2]{IM1}.
 
 \item If a cubic polynomial has a Siegel disk containing a post-critical point, then the corresponding map belongs to the bi-activity locus, and admits a disk of quasi-conformal deformations in the parameter space. It will be interesting to know if such parameters always lie in the support of the bifurcation measure.
 
 \item Our proof works only when the modulus of the basin of infinity in the repelling Ecalle cylinder is sufficiently large. It seems unlikely to have such an interval when the modulus is small, i.e., if the degree of the map or the period of the parabolic periodic point is large.
\end{enumerate}
\end{remark}

\bibliographystyle{alpha}
\bibliography{bifurcation}

\end{document}